\documentclass{article}
\usepackage[dvipsnames,usenames]{color}
\usepackage[colorlinks=true, urlcolor=NavyBlue, linkcolor=NavyBlue, citecolor=NavyBlue]{hyperref}
\usepackage{indentfirst}
\usepackage{amsmath,amsfonts,amsthm,amssymb}
\usepackage{mathrsfs}
\usepackage{amscd}
\usepackage{subfigure}
\usepackage{tikz}
\usetikzlibrary{arrows, patterns}
\usepackage{hyperref}

\allowdisplaybreaks[1]

\DeclareMathAlphabet{\mathsfsl}{OT1}{cmss}{m}{sl}

\newtheorem{thm}{Theorem}[section]
\newtheorem{lem}[thm]{Lemma}
\newtheorem{cor}[thm]{Corollary}
\newtheorem{prop}[thm]{Proposition}

\newtheorem{Thm}{Theorem}
\newtheorem{Prop}[Thm]{Proposition}

\theoremstyle{definition}
\newtheorem{defn}[thm]{Definition}

\newtheorem{rem}[thm]{Remark}

\numberwithin{equation}{section}

\newcommand{\Z}{\mathbb{Z}}

\newcommand{\spinc}{${\rm Spin}^c$ }

\newcommand\CFKi{CFK^\infty}
\newcommand\bbF{\mathbb{F}}

\begin{document}

\title{Four-ball genus bounds and a refinement of the Ozsv\'ath-Szab\'o tau-invariant}

\author{ {\large Jennifer HOM}\\{\normalsize Department of Mathematics, Columbia University}\\
{\normalsize 2990 Broadway, New York, NY 10027}\\{\small\it Emai\/l\/:\quad\rm hom@math.columbia.edu} \\ \\
\and {\large Zhongtao WU}\\{\normalsize Department of Mathematics, The Chinese Universiy of Hong Kong}\\
{\normalsize Lady Shaw Building, Shatin, Hong Kong}\\{\small\it Emai\/l\/:\quad\rm ztwu@math.cuhk.edu.hk}}

\date{}
\maketitle
\begin{abstract}

Based on work of Rasmussen \cite{RasThesis}, we construct a concordance invariant associated to the knot Floer complex, and exhibit examples in which this invariant gives arbitrarily better bounds on the $4$-ball genus than the Ozsv\'ath-Szab\'o $\tau$ invariant.

\end{abstract}

\section{Introduction}
The \emph{$4$-ball genus} of a knot $K \subset S^3$ is
\[ g_4(K) = \min \{ g(\Sigma) \mid \Sigma \textup{ smoothly embedded in } B^4 \textup{ with } \partial \Sigma = K \}, \]
where $g(\Sigma)$ denotes the genus of the surface $\Sigma$. The $4$-ball genus gives a lower bound on the unknotting number of a knot (that is, the minimal number of crossing changes needed to obtain the unknot). We say knots $K_1$ and $K_2$ are \emph{concordant} if $g_4(K_1 \# -K_2) = 0$, where $-K_2$ denotes the reverse of the mirror image of $K_2$.

In \cite{OSz4Genus}, Ozsv\'ath-Szab\'o defined a concordance invariant, $\tau$, that gives a lower bound for the $4$-ball genus of a knot. This invariant is sharp on torus knots, giving a new proof of the Milnor conjecture, originally proved by Kronheimer-Mrowka using gauge theory \cite{KronMrowka}

The knot Floer homology package \cite{OSknots, RasThesis} associates to a knot $K$ a $\Z \oplus \Z$-filtered chain complex over the ring $\bbF[U, U^{-1}]$, where $\bbF$ denotes the field of two elements and $U$ is a formal variable. We denote this complex $\CFKi(K)$. The invariant $\tau$ depends only on a single $\Z$-filtration, and forgets the module structure. By studying the module structure together with the full $\Z \oplus \Z$-filtration, we obtain a concordance invariant, $\nu^+$, which gives a better bound on the $4$-ball genus than $\tau$, in the sense that
\begin{equation}
	\tau(K) \leq \nu^+(K) \leq g_4(K).
\end{equation}
Moreover, the gap between $\tau$ and $\nu^+$ can be made arbitrarily large.

\begin{Thm}\label{betterbound}
For any positive integer $p$, there exists a knot $K$ with $\tau(K) \geq 0$ and
\[ \tau(K)+p \leq \nu^+(K) = g_4(K). \]
\end{Thm}

\begin{rem}
The invariant $\nu^+$ is closely related to the sequence of local $h$ invariants of Rasmussen \cite[Section 7]{RasThesis}, which Rasmussen uses to give bounds on the $4$-ball genus; indeed, $\nu^+$ corresponds to the first place in the sequence where a zero appears.
\end{rem}

\noindent In Proposition \ref{prop:sig}, we also show that the gap between $\nu^+$ and the knot signature can be made arbitrarily large.

In the case of alternating knots (or, more generally, quasi-alternating knots), the invariant $\nu^+$ is completely determined by the signature of the knot.
\begin{Thm}\label{thm:QA}
Let $K \subset S^3$ be a quasi-alternating knot.  Then, $$\nu^+(K)=\left\{
\begin{array}{ll}
0&\text{if } \sigma(K)\ge0,\\
-\frac{\sigma(K)}{2}&\text{if }\sigma(K)<0.
\end{array}
\right.$$
\end{Thm}

We also have the following result when $K$ is strongly quasipositive. See \cite{Heddenpositivity} for background on strongly quasipositive knots.

\begin{Prop}
If $K$ is strongly quasipositive, then
\[ \nu^+(K) = \tau(K) = g_4(K) = g(K). \]
\end{Prop}

\begin{proof}
\cite[Theorem 1.2]{Heddenpositivity} states that $\tau(K)=g_4(K)=g(K)$ if and only if $K$ is strongly quasipositive. Since $\tau(K) \leq \nu^+(K) \leq g_4(K)$, the result follows.
\end{proof}

\vspace{5pt}\noindent{\bf Organization.} In Section \ref{sec:nuplus}, we define the invariant $\nu^+$ and prove various properties. In Section \ref{sec:4ball}, we construct an infinite family of knots in order to prove Theorem \ref{betterbound}. Throughout, we work over $\bbF=\Z/2\Z$.

\vspace{5pt}\noindent{\bf Acknowledgements.} The first author was partially supported by NSF grant DMS-1307879. The second author would like to thank Hiroshi Goda for helpful email communications.

\section{The invariant $\nu^+$}\label{sec:nuplus}

Heegaard Floer homology, introduced by
Ozsv\'ath and Szab\'o \cite{OS3manifolds1}, is an invariant  for closed oriented Spin$^c$ $3$--manifolds $(Y,\mathfrak s)$,
  taking the form of a collection of related homology groups:  $\widehat{HF}(Y,\mathfrak s)$, $HF^{\pm}(Y,\mathfrak s)$, and $HF^\infty(Y,\mathfrak s)$.
There is a $U$--action on the Heegaard Floer homology groups $HF^{\pm}$ and $HF^\infty$.
When $\mathfrak s$ is torsion, there is an absolute Maslov $\mathbb Q$--grading on the Heegaard Floer homology groups. The $U$--action decreases the grading by $2$.

For a rational homology $3$--sphere $Y$ with a Spin$^c$ structure $\mathfrak s$, $HF^+(Y,\mathfrak s)$ can be decomposed as the direct sum of two groups: the first group is the image of $HF^\infty(Y,\mathfrak s)\cong\mathbb \bbF[U,U^{-1}]$ in $HF^+(Y,\mathfrak s)$,
which is isomorphic to $\mathcal T^+=\bbF[U,U^{-1}]/U\bbF[U]$, and its minimal absolute  $\mathbb{Q}$--grading is an invariant of $(Y,\mathfrak s)$, denoted by $d(Y,\mathfrak s)$, the {\it correction term} \cite{OSabsgr}; the second group is the quotient modulo the above image and is denoted by $HF_{\mathrm{red}}(Y,\mathfrak s)$.  Altogether, we have $$HF^+(Y,\mathfrak s)=\mathcal{T}^+\oplus HF_{\mathrm{red}}(Y,\mathfrak s).$$

We briefly recall the large $N$ surgery formula of \cite[Theorem 4.4]{OSknots}. We use the notation of \cite{NiWu}. Let $\CFKi(K)$ denote the knot Floer complex of $K$, which takes the form of a $\Z \oplus \Z$-filtered, $\Z$-graded chain complex over $\bbF[U, U^{-1}]$. The $U$-action lowers each filtration by one. We will be particularly interested in the quotient complexes
\[ A^+_k = C\{ \max \{i, j-k\}  \geq 0\}  \quad \textup{ and } \quad B^+ = C\{ i \geq 0\} \]
where $i$ and $j$ refer to the two filtrations. The complex $B^+$ is isomorphic to $CF^+(S^3)$. There is a map
\[ v^+_k: A^+_k \rightarrow B^+ \]
defined by projection. One can also define a map
\[ h^+_k: A^+_k \rightarrow B^+ \]
defined by projection to $C\{j \geq k\}$, followed by shifting to $C\{j \geq 0\}$ via the $U$-action, and concluding with a chain homotopy equivalence between $C\{j \geq 0\}$ and $C\{i \geq 0\}$. These maps correspond to the maps induced on $HF^+$ by the two handle cobordism from $S^3_N(K)$ to $S^3$ \cite[Theorem 4.4]{OSknots}.

Similarly, one can consider the subquotient complexes
\[ \widehat{A}_k = C\{ \max \{i, j-k\} =0 \} \quad \textup{ and } \widehat{B} = C\{ i = 0\} \cong \widehat{CF}(S^3) \]
and the maps
\[ \widehat{v}_k: \widehat{A}_k \rightarrow \widehat{B} \quad \textup{ and } \quad  \widehat{h}_k: \widehat{A}_k \rightarrow \widehat{B}. \]

The invariant $\tau$ is defined in \cite{OSz4Genus} to be
\[ \tau(K) = \min \{ k \in \Z \mid \iota_k \textup{ induces a nontrivial map on homology}\}, \]
where $\iota_k : C\{ i=0, j \leq k\} \rightarrow \widehat{CF}(S^3)$ denotes inclusion.
A slightly stronger concordance invariant, $\nu$, is defined in \cite[Definition 9.1]{OSrational} to be
$$\nu(K)=\mathrm{min} \{k\in \Z \,|\, \widehat{v}_k: \widehat{A}_k \rightarrow \widehat{CF}(S^3) \;\; \text{induces a nontrivial map in homology} \}.$$
The invariant $\nu(K)$ gives a lower bound for $g_4(K)$ and is equal to either $\tau(K)$ or $\tau(K)+1$; in particular, in many cases $\nu$ gives a better $4$-ball genus than $\tau$.

We can further refine these bounds by considering maps on $CF^+$ rather than $\widehat{CF}$.

\begin{defn}
Define $\nu^+(K)$ by
$$\nu^+(K)=\mathrm{min} \{k\in \Z \,|\, v^+_k: A^+_k \rightarrow CF^+(S^3), \;\; v^+_k(1)=1 \}.$$
Here, $1$ denotes the lowest graded generator of the subgroup $\mathcal T^+$ in the homology of the complex.
\end{defn}

According to \cite{NiWu}, the definition of $\nu^+(K)$ is equivalent to the smallest $k$ such that $V_k=0$, where $V_k$ is the $U$-exponent of $v^+_k$ at sufficiently high gradings. We can define $H_k$ similarly in terms of $h^+_k$. By \cite[Equation (13)]{NiWurational} and \cite[Lemma 2.5]{HLZ}, the $V_k$'s and $H_k$'s satisfy
\begin{align}
	H_k &= V_{-k} \label{eqn:HV} \\
	H_k &= V_k + k \label{eqn:VHk} \\
	V_k -1 &\leq V_{k+1} \leq V_k \label{eqn:monotonicity}
\end{align}
and are related to the correction terms in the surgery formula \cite[Proposition 1.6]{NiWu}:

\begin{prop}\label{prop:Corr}
Suppose $p,q>0$, and fix $0\le i\le p-1$. Then
\begin{equation}\label{Corr}
d(S^3_{p/q}(K),i)=d(L(p,q),i)-2\max\{V_{\lfloor\frac{i}q\rfloor},H_{\lfloor\frac{i-p}q\rfloor}\}.
\end{equation}
\end{prop}

We have the following properties for $\nu^+$.

\begin{prop}
The invariant $\nu^+$ satisfies:
\begin{enumerate}
\item \label{it:invt}
$\nu^+$ is a smooth concordance invariant.
\item \label{it:V0}
$\nu^+(K) \geq 0$, and the equality holds if and only if $V_0=0$.
\item \label{it:tau}
$\nu^+(K) \geq \nu(K) \geq \tau(K)$.
\end{enumerate}
\end{prop}

\begin{proof}
To see \ref{it:invt}, note that $V$'s are determined by the $d$-invariants of the surgered manifolds $S^3_n(K)$ \cite[Proposition 2.11]{NiWu}, and the $d$-invariants are concordance invariants.  To see \ref{it:V0}, note that $V_{-1}>H_{-1}=V_{1}\geq 0$ by Equations (\ref{eqn:HV}) and (\ref{eqn:VHk}).  To see \ref{it:tau}, chase the commutative diagram
$$
\begin{CD}
\widehat A_k&@>j_A>>&A^+_k\\
@V\widehat v_kVV&&@Vv^+_kVV\\
\widehat B&@>j_B>>&B^+.
\end{CD}
$$
\end{proof}

The $\nu^+$ invariant can be computed explicitly for quasi-alternating knots, a generalization of alternating knots introduced in \cite{ManolescuOzsvath}. In fact, Theorem \ref{thm:QA} states that $\nu^+$ is completely determined by the signature of the knot, just as the $\tau$ invariant:
$$\nu^+(K)=\left\{
\begin{array}{ll}
0&\text{if } \sigma(K)\ge0,\\
-\frac{\sigma(K)}{2}&\text{if }\sigma(K)<0.
\end{array}
\right.$$

\begin{proof}[Proof of Theorem \ref{thm:QA}]
Let $K$ be quasi-alternating. By \cite[Corollary 1.5]{OSzAlt} and \cite[Theorem 2]{ManolescuOzsvath}, $d(S^3_1(K))=0$ when $\sigma(K)\ge 0$.  This proves that $\nu^+(K)=0$ when $\sigma(K)\ge 0$.  On the other hand, the proof of Theorem 1.4 of \cite{OSzAlt}, together with \cite[Theorem 2]{ManolescuOzsvath}, implies that for any $s>0$,
$$H_{\le s+\frac{\sigma}{2}-2}(A_s^+)\cong HF^+_{\le s+\frac{\sigma}{2}-2}(S^3).$$
In particular, if we let $s=-\sigma/2$ when $\sigma(K)<0$, then $$H_{\leq-2}(A_s^+)\cong HF^+_{\leq-2}(S^3)\cong 0.$$
Here, the gradings of the homology of both sides are inherited from the grading on $CFK^\infty(K)$. Thus, the generator of $\mathcal{T}^+ \subset H_*(A^+_s)$ has grading $-2V_s$. In light of the vanishing of the homology group $H_{\leq-2}(A_s^+)$, we must have $V_s=0$. So
$$\nu^+(K)\leq s= -\sigma(K)/2$$ from the definition.  We also know that $$\nu^+(K)\geq \tau(K)=-\sigma(K)/2$$ for a quasi-alternating knot $K$.  Hence, $\nu^+(K)=-\sigma(K)/2$.
\end{proof}
%C\{\max(i,j-s)\ge 0\}

Next, we show that $\nu^+$ also give a lower bound for the four-ball genus of a knot.

\begin{prop}
$\nu^+(K) \leq g_4(K)$
\end{prop}

\begin{proof}
This follows from \cite[Corollary 7.4]{RasThesis}.  The function $h_k(K)$ in \cite{RasThesis} is the same as $V_k$ in \cite{NiWu}.
\end{proof}

\begin{rem}
\cite[Corollary 7.4]{RasThesis} states that $g_4(K) \geq V_k+k$ for all $k\leq g_4(K)$, so one might wonder if other $V_k$'s can give stronger $4$-ball genus bounds. However, since $V_k-1  \leq V_{k+1}  \leq V_k$, it follows that $\nu^+$ is the best $4$-ball genus bound obtainable from the sequence of $V_k$'s.
\end{rem}

\section{Four-ball genus bound}\label{sec:4ball}

In this section, we exhibit some examples of knots whose $\nu^+$ invariant is arbitrarily better than the corresponding $\tau$ invariant.  Hence, the $\nu^+$ invariant indeed gives us significantly improved four-ball genus bound for some particular knots.  We will show that for any integer $n \geq 2$, there exists a knot $K$ with $\tau(K) \geq 0$ and
\[ \tau(K)+n = \nu^+(K) = g_4(K).\]

Let $K_{p, q}$ denote the $(p, q)$-cable of $K$, where $p$ denotes the longitudinal winding. Without loss of generality, we will assume throughout that $p>0$. Let $T_{p, q}$ denote the $(p, q)$-torus knot (that is, the $(p, q)$-cable of the unknot), and $T_{p,q; m, n}$ the $(m, n)$-cable of $T_{p, q}$. We begin with a single example of a knot for which $\nu^+$ gives a better $4$-ball genus bound than $\tau$.

\begin{prop}
Let $K$ be the knot  $T_{2,9} \# -T_{2,3; 2,5}$. We have
\[ \tau(K) = 0, \qquad \nu(K) = 1, \quad \textup{ and } \quad \nu^+(K) = 2. \]
\end{prop}

\begin{proof}
The torus knot $T_{2,9}$ is an $L$-space knot, as is $T_{2,3; 2,5}$ \cite[Theorem 1.10]{HeddencablingII}, so their knot Floer complexes are completely determined by their Alexander polynomials \cite[Theorem 1.2]{OSlens} (cf. \cite[Remark 6.6]{Homsmooth}). We have that
\[ \Delta_{T_{2,9}}(T) = t^8 -t^7 + t^6 - t^5 + t^4 - t^3 + t^2 -t +1 \]
and
\begin{align*}
	\Delta_{T_{2,3; 2,5}}(t) &= \Delta_{T_{2,3}}(t^2) \cdot \Delta_{T_{2,5}}(t) \\
				&= t^8 -t^7 + t^4 -t +1.
\end{align*}
Furthermore, we have that $\CFKi(-K) \cong \CFKi(K)^*$ \cite[Section 3.5]{OSknots}, where $\CFKi(K)^*$ denotes the dual of $\CFKi(K)$. Thus, $\CFKi(-T_{2,3; 2,5})$ is generated over $\bbF[U, U^{-1}]$ by
\[ [y_0, 0, -4], \quad [y_1, -1, -4], \quad [y_1, -1, -1], \quad [y_3, -3, -1], \quad [y_4, -4, 0], \]
where we write $[y, i, j]$ to denote that the generator $y$ has filtration level $(i, j)$. The differential is given by
\begin{align*}
	\partial y_0 &= y_1 \\
	\partial y_2 &= y_1 + y_3 \\
	\partial y_4 &= y_3.
\end{align*}
The complex $\CFKi(T_{2,9})$ is generated by
\begin{equation*}
  \begin{gathered}
  \ [x_0, 0, 4], \quad [x_1, 1, 4], \quad [x_2, 1, 3], \quad [x_3, 2, 3], \quad [x_4, 2, 2], \\
 \quad [x_5, 3, 2], \quad [x_6, 3, 1], \quad [x_7, 4, 1], \quad [x_8, 4, 0].
  \end{gathered}
\end{equation*}
The differential is given by
\begin{align*}
	\partial x_1 &= x_0 +x_2 \\
	\partial x_3 &= x_2 + x_4 \\
	\partial x_5 &= x_4 + x_6 \\
	\partial x_7 &= x_6 + x_8.	
\end{align*}

The complexes $\CFKi(-T_{2,3; 2,5})$ and $\CFKi(T_{2,9})$ are depicted in Figures \ref{fig:T2325} and \ref{fig:T29}, respectively. (More precisely, $\CFKi$ consists of the complexes pictured tensored with $\bbF[U, U^{-1}]$, where $U$ lowers $i$ and $j$ each by $1$.) In particular, we see that $\tau(-T_{2,3; 2,5})= -4$ since $y_0$ generates the vertical homology, and that $\tau(T_{2,9})=4$ since $x_0$ generates the vertical homology. Since $\tau$ is additive under connected sum, it follows that
\[ \tau(-T_{2,3; 2,5} \# T_{2,9})=0, \]
as desired.

\begin{figure}[htb!]
\vspace{5pt}
\centering
\begin{tikzpicture}[scale=0.8]

	\draw[step=1, black!20!white, very thin] (-4.9, -4.9) grid (0.9, 0.9);
	
	\begin{scope}[thin, black!60!white]
		\draw [<->] (-5, 0) -- (1, 0);
		\draw [<->] (0, -5) -- (0, 1);
	\end{scope}

	\filldraw (-4, 0) circle (2pt) node[] (a){};
	\filldraw (-4, -1) circle (2pt) node[] (b){};
	\filldraw (-1, -1) circle (2pt) node[] (c){};
	\filldraw (-1, -4) circle (2pt) node[] (d){};
	\filldraw (0, -4) circle (2pt) node[] (e){};
	\draw [very thick, ->] (a) -- (b);
	\draw [very thick, ->] (c) -- (b);
	\draw [very thick, ->] (c) -- (d);
	\draw [very thick, ->] (e) -- (d);
	\node [right] at (a) {$y_4$};
	\node [left] at (b) {$y_3$};
	\node [right] at (c) {$y_2$};
	\node [left] at (d) {$y_1$};
	\node [right] at (e) {$y_0$};
\end{tikzpicture}
\caption{$\CFKi(-T_{2,3; 2,5})$}
\label{fig:T2325}
\end{figure}
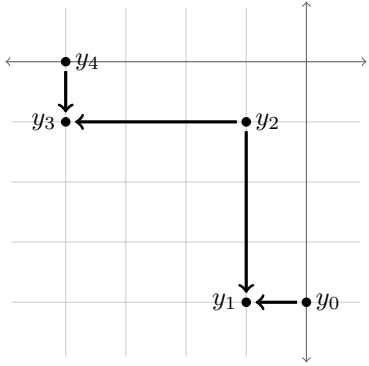

\begin{figure}[htb!]
\vspace{5pt}
\centering
\begin{tikzpicture}[scale=0.8]

	\draw[step=1, black!20!white, very thin] (-0.9, -0.9) grid (4.9, 4.9);
	
	\begin{scope}[thin, black!60!white]
		\draw [<->] (-1, 0) -- (5, 0);
		\draw [<->] (0, -1) -- (0, 5);
	\end{scope}
	
	\filldraw (0, 4) circle (2pt) node[] (a){};
	\filldraw (1, 4) circle (2pt) node[] (b){};
	\filldraw (1, 3) circle (2pt) node[] (c){};
	\filldraw (2, 3) circle (2pt) node[] (d){};
	\filldraw (2, 2) circle (2pt) node[] (e){};
	\filldraw (3, 2) circle (2pt) node[] (f){};
	\filldraw (3, 1) circle (2pt) node[] (g){};
	\filldraw (4, 1) circle (2pt) node[] (h){};
	\filldraw (4, 0) circle (2pt) node[] (i){};
	\draw [very thick, <-] (a) -- (b);
	\draw [very thick, <-] (c) -- (b);
	\draw [very thick, <-] (c) -- (d);
	\draw [very thick, <-] (e) -- (d);
	\draw [very thick, <-] (e) -- (f);
	\draw [very thick, <-] (g) -- (f);
	\draw [very thick, <-] (g) -- (h);
	\draw [very thick, <-] (i) -- (h);
	\node [left] at (a) {$x_0$};
	\node [right] at (b) {$x_1$};
	\node [left] at (c) {$x_2$};
	\node [right] at (d) {$x_3$};
	\node [left] at (e) {$x_4$};
	\node [right] at (f) {$x_5$};
	\node [left] at (g) {$x_6$};
	\node [right] at (h) {$x_7$};
	\node [left] at (i) {$x_8$};
\end{tikzpicture}
\caption{$\CFKi(T_{2,9})$}
\label{fig:T29}
\end{figure}
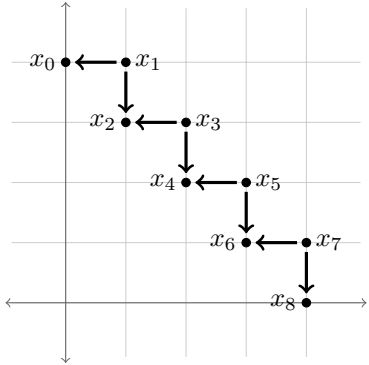

The knot Floer complex satisfies a K\"unneth formula \cite[Theorem 7.1]{OSknots}:
	\[ \CFKi(K_1 \# K_2) \cong \CFKi(K_1) \otimes_{\bbF[U, U^{-1}]} \CFKi(K_2). \]
In particular, we may compute $\CFKi(T_{2,9} \# -T_{2,3; 2, 5})$ as the tensor product of $\CFKi(T_{2,9})$ and $\CFKi(-T_{2,3; 2,5})$ , where
\[ [x, i, j] \otimes [y, k, \ell] = [xy, i+k, j+\ell]. \]

The generators, filtration levels, and differentials in the tensor product are listed below.

\begin{align*}
	\partial [x_0y_0, 0, 0] &= x_0y_1 \\			
	\partial [x_1y_0, 1, 0] &= x_1y_1+x_0y_0 +x_2y_0 \\
	\partial [x_2y_0, 1, -1] &= x_2y_1  \\
	\partial [x_3y_0, 2, -1] &= x_3y_1+x_2y_0+x_4y_0 \\
	\partial [x_4y_0, 2, -2] &= x_4y_1 \\
	\partial [x_5y_0, 3, -2] &= x_5y_1+x_4y_0 +x_6y_0	\\
	\partial [x_6y_0, 3, -3] &= x_6y_1 \\
	\partial [x_7y_0, 4, -3] &= x_7y_1+x_6y_0 +x_8y_0	\\
	\partial [x_8y_0, 4, -4] &= x_8y_1 \\
	\partial [x_0y_1, -1, 0] &= 0 \\
	\partial [x_1y_1, 0, 0] &= x_0y_1+x_2y_1 	\\
	\partial [x_2y_1, 0, -1] &= 0 \\
	\partial [x_3y_1, 1, -1] &= x_2y_1+x_4y_1 \\
	\partial [x_4y_1, 1, -2] &= 0 \\
	\partial [x_5y_1, 2, -2] &= x_4y_1+x_6y_1 \\
	\partial [x_6y_1, 2, -3] &= 0 \\
	\partial [x_7y_1, 3, -3] &= x_6y_1+x_8y_1 \\
	\partial [x_8y_1, 3, -4] &= 0 \\
	\partial [x_0y_2, -1, 3] &= x_0y_1+x_0y_3 \\
	\partial [x_1y_2, 0, 3] &= x_1y_1+x_1y_3 +x_0y_2+x_2y_2\\
	\partial [x_2y_2, 0, 2] &= x_2y_1+x_2y_3 \\
	\partial [x_3y_2, 1, 2] &= x_3y_1+x_3y_3 +x_2y_2+x_4y_2\\
	\partial [x_4y_2, 1, 1] &= x_4y_1+x_4y_3 \\
	\partial [x_5y_2, 2, 1] &= x_5y_1+x_5y_3 +x_4y_2+x_6y_2\\
	\partial [x_6y_2, 2, 0] &= x_6y_1+x_6y_3 \\
	\partial [x_7y_2, 3, 0] &= x_7y_1+x_7y_3 +x_6y_2+x_8y_2\\
	\partial [x_8y_2, 3, -1] &= x_8y_1+x_8y_3 \\
	\partial [x_0y_3, -4, 3] &= 0 \\
	\partial [x_1y_3, -3, 3] &= x_0 y_3+x_2y_3 \\
	\partial [x_2y_3, -3, 2] &= 0 \\
	\partial [x_3y_3, -2, 2] &= x_2 y_3+x_4y_3 \\
	\partial [x_4y_3, -2, 1] &= 0 \\
	\partial [x_5y_3, -1, 1] &= x_4 y_3+x_6y_3 \\
	\partial [x_6y_3, -1, 0] &= 0 \\
	\partial [x_7y_3, 0, 0] &= x_6 y_3+x_8y_3 \\
	\partial [x_8y_3, 0, -1] &= 0 \\
	\partial [x_0y_4, -4, 4] &= x_0 y_3 \\
	\partial [x_1y_4, -3, 4] &= x_1 y_3+x_0y_4+x_2y_4 \\
	\partial [x_2y_4, -3, 3] &= x_2 y_3 \\
	\partial [x_3y_4, -2, 3] &= x_3 y_3+x_2y_4+x_4y_4 \\
	\partial [x_4y_4, -2, 2] &= x_4 y_3 \\
	\partial [x_5y_4, -1, 2] &= x_5 y_3+x_4y_4+x_6y_4 \\
	\partial [x_6y_4, -1, 1] &= x_6 y_3 \\
	\partial [x_7y_4, 0, 1] &= x_7 y_3+x_6y_4+x_8y_4 \\
	\partial [x_8y_4, 0, 0] &= x_8 y_3 \\	
\end{align*}

We perform the following change of basis on $\CFKi(T_{2,9} \# -T_{2,3; 2, 5})$. In the linear combinations below, we have ordered the terms so that the first basis element has the greatest filtration and thus determines the filtration level of the linear combination.
\begin{align*}
	z_0 &= x_0 y_0 \\
	z_1 &= x_0 y_1 \\
	z_2 &= x_0y_2 + x_1 y_3 + x_3 y_3 + x_4 y_4 \\
	z_3 &= x_1 y_2 \\
	z_4 &= x_2y_2 + x_3 y_3 + x_1 y_1 + x_4 y_4 \\
	z_5 &= x_3 y_2 + x_5 y_4 + x_1 y_0 \\
	z_6 &= x_4 y_2 + x_5 y_3 + x_3 y_1 + x_6 y_4 + x_2 y_0 \\
	z_7 &= x_5 y_2 + x_7 y_4 + x_3 y_0 \\
	z_8 &= x_6 y_2 + x_7 y_3 + x_5 y_1 + x_4 y_0 \\
	z_9 &= x_7 y_2 \\
	z_{10} &= x_8 y_2 + x_7 y_1 + x_4 y_0 + x_5 y_1 \\
	z_{11} &= x_8 y_3 \\
	z_{12} &= x_8 y_4 \\
	w^i_0 &= x_{2i+1} y_4 & i&=0, 1, 2, 3 \\
	w^i_1 &= x_{2i} y_4  & i&=0, 1, 2, 3 \\
	w^i_2 &= x_{2i} y_3 & i&=0, 1, 2, 3 \\
	w^i_3 &=  x_{2i+1} y_3 + x_{2i+2} y_4 & i&=0, 1, 2, 3 \\
	w^{i+4}_0 &= x_{2i+1} y_0 & i&=0, 1, 2, 3 \\
	w^{i+4}_1 &=  x_{2i+1} y_1 + x_{2i} y_0 & i&=0, 1, 2, 3 \\
	w^{i+4}_2 &= x_{2i+2} y_1 & i&=0, 1, 2, 3 \\
	w^{i+4}_3 &=  x_{2i+2} y_0 & i&=0, 1, 2, 3.
\end{align*}
See Figure \ref{fig:aftercob}.

\begin{figure}[htb!]
\vspace{5pt}
\centering
\begin{tikzpicture}

	\draw[step=1, black!30!white, very thin] (-4.9, -4.9) grid (4.9, 4.9);
	
	\begin{scope}[thin, black!60!white]
		\draw [<->] (-5, 0) -- (5, 0);
		\draw [<->] (0, -5) -- (0, 5);
	\end{scope}
	
	\filldraw (-0.2, 0) circle (2pt) node[] (a){};
	\filldraw (-1, 0) circle (2pt) node[] (b){};
	\filldraw (-1, 3) circle (2pt) node[] (c){};
	\filldraw (0, 3) circle (2pt) node[] (d){};
	\filldraw (0, 2) circle (2pt) node[] (e){};
	\filldraw (1, 2) circle (2pt) node[] (f){};
	\filldraw (1, 1) circle (2pt) node[] (g){};
	\filldraw (2, 1) circle (2pt) node[] (h){};
	\filldraw (2, 0) circle (2pt) node[] (i){};
	\filldraw (3, 0) circle (2pt) node[] (j){};
	\filldraw (3, -1) circle (2pt) node[] (k){};
	\filldraw (0, -1) circle (2pt) node[] (l){};
	\filldraw (0, -0.2) circle (2pt) node[] (m){};
	\draw [very thick, <-] (b) -- (a);
	\draw [very thick, <-] (b) -- (c);
	\draw [very thick, <-] (c) -- (d);
	\draw [very thick, <-] (e) -- (d);
	\draw [very thick, <-] (e) -- (f);
	\draw [very thick, <-] (g) -- (f);
	\draw [very thick, <-] (g) -- (h);
	\draw [very thick, <-] (i) -- (h);
	\draw [very thick, <-] (i) -- (j);
	\draw [very thick, <-] (k) -- (j);
	\draw [very thick, <-] (l) -- (k);
	\draw [very thick, <-] (l) -- (m);
	\draw [very thick, <-] (b) -- (e);
	\draw [very thick, <-] (a) -- (f);
	\draw [very thick, <-] (l) -- (i);
	\draw [very thick, <-] (m) -- (h);
	\node [below, xshift=-5pt, yshift=3pt] at (a) {$z_0$};
	\node [left] at (b) {$z_1$};
	\node [left] at (c) {$z_2$};
	\node [right] at (d) {$z_3$};
	\node [left] at (e) {$z_4$};
	\node [right] at (f) {$z_5$};
	\node [left] at (g) {$z_6$};
	\node [right] at (h) {$z_7$};
	\node [left, yshift=4pt] at (i) {$z_8$};
	\node [right] at (j) {$z_9$};
	\node [right] at (k) {$z_{10}$};
	\node [left] at (l) {$z_{11}$};
	\node [left, xshift=3pt, yshift=-5pt] at (m) {$z_{12}$};
	
	\filldraw (-3.1, 4) circle (2pt) node[] (a1){};
	\filldraw (-3.1, 3.1) circle (2pt) node[] (b1){};
	\filldraw (-4, 4) circle (2pt) node[] (c1){};
	\filldraw (-4, 3.1) circle (2pt) node[] (d1){};
	\draw [very thick, <-] (b1) -- (a1);
	\draw [very thick, <-] (c1) -- (a1);
	\draw [very thick, <-] (d1) -- (b1);
	\draw [very thick, <-] (d1) -- (c1);
	\node [above] at (a1) {$w^0_0$};
	\node [left] at (c1) {$w^0_1$};
	\node [left] at (d1) {$w^0_2$};
	\node [above left, xshift=1.5pt, yshift=-1.5pt] at (b1) {$w^0_3$};
	
	\filldraw (-2.1, 2.9) circle (2pt) node[] (a2){};
	\filldraw (-2.1, 2.1) circle (2pt) node[] (b2){};
	\filldraw (-2.9, 2.9) circle (2pt) node[] (c2){};
	\filldraw (-2.9, 2.1) circle (2pt) node[] (d2){};
	\draw [very thick, <-] (b2) -- (a2);
	\draw [very thick, <-] (c2) -- (a2);
	\draw [very thick, <-] (d2) -- (b2);
	\draw [very thick, <-] (d2) -- (c2);
	\node [above] at (a2) {$w^1_0$};
	\node [left, xshift=2pt, yshift=-5pt] at (c2) {$w^1_1$};
	\node [left] at (d2) {$w^1_2$};
	\node [above left, xshift=1.5pt, yshift=-1.5pt] at (b2) {$w^1_3$};
	
	\filldraw (-1.1, 1.9) circle (2pt) node[] (a3){};
	\filldraw (-1.1, 1.1) circle (2pt) node[] (b3){};
	\filldraw (-1.9, 1.9) circle (2pt) node[] (c3){};
	\filldraw (-1.9, 1.1) circle (2pt) node[] (d3){};
	\draw [very thick, <-] (b3) -- (a3);
	\draw [very thick, <-] (c3) -- (a3);
	\draw [very thick, <-] (d3) -- (b3);
	\draw [very thick, <-] (d3) -- (c3);
	\node [above, xshift=-5pt] at (a3) {$w^2_0$};
	\node [left, xshift=2pt, yshift=-5pt] at (c3) {$w^2_1$};
	\node [left] at (d3) {$w^2_2$};
	\node [above left, xshift=1.5pt, yshift=-1.5pt] at (b3) {$w^2_3$};
	
	\filldraw (-0.15, 0.85) circle (2pt) node[] (a4){};
	\filldraw (-0.15, 0.15) circle (2pt) node[] (b4){};
	\filldraw (-0.85, 0.85) circle (2pt) node[] (c4){};
	\filldraw (-0.85, 0.15) circle (2pt) node[] (d4){};
	\draw [very thick, <-] (b4) -- (a4);
	\draw [very thick, <-] (c4) -- (a4);
	\draw [very thick, <-] (d4) -- (b4);
	\draw [very thick, <-] (d4) -- (c4);
	
	\filldraw (0.85, -0.15) circle (2pt) node[] (a5){};
	\filldraw (0.85, -0.85) circle (2pt) node[] (b5){};
	\filldraw (0.15, -0.15) circle (2pt) node[] (c5){};
	\filldraw (0.15, -0.85) circle (2pt) node[] (d5){};
	\draw [very thick, <-] (b5) -- (a5);
	\draw [very thick, <-] (c5) -- (a5);
	\draw [very thick, <-] (d5) -- (b5);
	\draw [very thick, <-] (d5) -- (c5);

	\filldraw (1.9, -1.1) circle (2pt) node[] (a6){};
	\filldraw (1.9, -1.9) circle (2pt) node[] (b6){};
	\filldraw (1.1, -1.1) circle (2pt) node[] (c6){};
	\filldraw (1.1, -1.9) circle (2pt) node[] (d6){};
	\draw [very thick, <-] (b6) -- (a6);
	\draw [very thick, <-] (c6) -- (a6);
	\draw [very thick, <-] (d6) -- (b6);
	\draw [very thick, <-] (d6) -- (c6);
	\node [right, yshift=-4pt] at (a6) {$w^5_0$};
	\node [below right, xshift=-2pt, yshift=2pt] at (c6) {$w^5_1$};
	\node [left] at (d6) {$w^5_2$};
	\node [right, yshift=4pt] at (b6) {$w^5_3$};
	
	\filldraw (2.9, -2.1) circle (2pt) node[] (a7){};
	\filldraw (2.9, -2.9) circle (2pt) node[] (b7){};
	\filldraw (2.1, -2.1) circle (2pt) node[] (c7){};
	\filldraw (2.1, -2.9) circle (2pt) node[] (d7){};
	\draw [very thick, <-] (b7) -- (a7);
	\draw [very thick, <-] (c7) -- (a7);
	\draw [very thick, <-] (d7) -- (b7);
	\draw [very thick, <-] (d7) -- (c7);
	\node [right] at (a7) {$w^6_0$};
	\node [below right, xshift=-2pt, yshift=2pt] at (c7) {$w^6_1$};
	\node [left] at (d7) {$w^6_2$};
	\node [right, yshift=4pt] at (b7) {$w^6_3$};

	\filldraw (3.9, -3.1) circle (2pt) node[] (a8){};
	\filldraw (3.9, -3.9) circle (2pt) node[] (b8){};
	\filldraw (3.1, -3.1) circle (2pt) node[] (c8){};
	\filldraw (3.1, -3.9) circle (2pt) node[] (d8){};
	\draw [very thick, <-] (b8) -- (a8);
	\draw [very thick, <-] (c8) -- (a8);
	\draw [very thick, <-] (d8) -- (b8);
	\draw [very thick, <-] (d8) -- (c8);
	\node [right] at (a8) {$w^7_0$};
	\node [below right, xshift=-2pt, yshift=2pt] at (c8) {$w^7_1$};
	\node [left] at (d8) {$w^7_2$};
	\node [right] at (b8) {$w^7_3$};
	
\end{tikzpicture}
\caption{$\CFKi(T_{2,9} \# -T_{2,3; 2, 5})$ after a change of basis}
\label{fig:aftercob}
\end{figure}
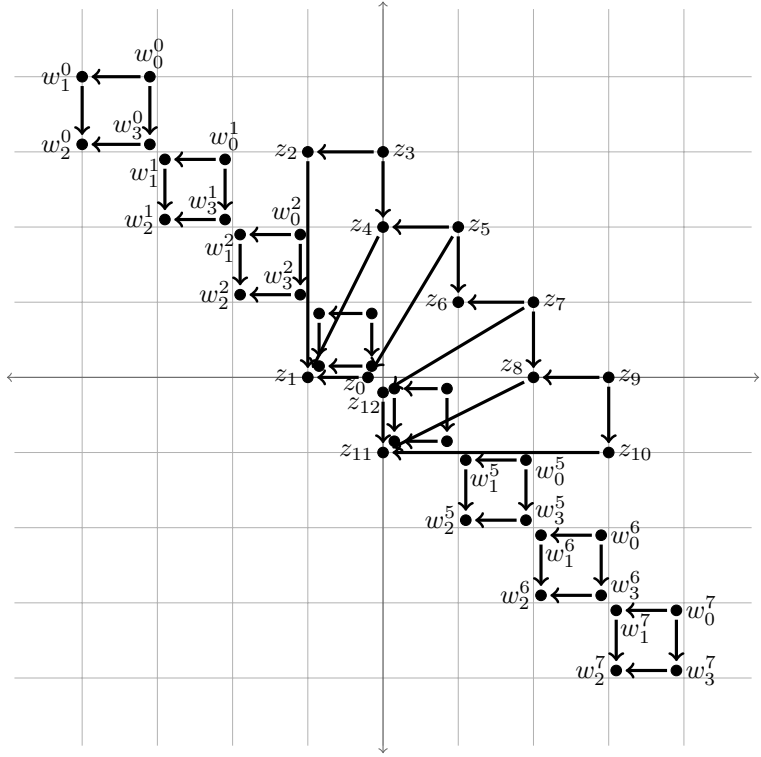

Notice that the basis elements $\{ z_i \}_{i=0}^{12}$ generate a direct summand $C$ of $\CFKi(T_{2,9} \# -T_{2,3; 2, 5})$. See Figure \ref{fig:summand}. Since the total homology of this summand is non-zero, this summand determines both $\nu$ and $\nu^+$. We write $\widehat{A}_s$ and $A^+_s$ to refer to the associated subquotient complexes of $C$.

The vertical homology of $C$ is generated by $z_0$. The generator $z_0$ in $C\{i=0\}$ is not the image of any cycle in $\widehat{A}_0$. On the other hand, $z_0$ is non-zero in $H_*(\widehat{A}_1)$. Hence $\nu(T_{2,9} \# -T_{2,3; 2, 5})=1$.

\begin{figure}[htb!]
\vspace{5pt}
\centering
\begin{tikzpicture}

	\draw[step=1, black!30!white, very thin] (-2.9, -2.9) grid (3.9, 3.9);
	
	\begin{scope}[thin, black!60!white]
		\draw [<->] (-3, 0) -- (4, 0);
		\draw [<->] (0, -3) -- (0, 4);
	\end{scope}
	
	\filldraw (-0.2, 0) circle (2pt) node[] (a){};
	\filldraw (-1, 0) circle (2pt) node[] (b){};
	\filldraw (-1, 3) circle (2pt) node[] (c){};
	\filldraw (0, 3) circle (2pt) node[] (d){};
	\filldraw (0, 2) circle (2pt) node[] (e){};
	\filldraw (1, 2) circle (2pt) node[] (f){};
	\filldraw (1, 1) circle (2pt) node[] (g){};
	\filldraw (2, 1) circle (2pt) node[] (h){};
	\filldraw (2, 0) circle (2pt) node[] (i){};
	\filldraw (3, 0) circle (2pt) node[] (j){};
	\filldraw (3, -1) circle (2pt) node[] (k){};
	\filldraw (0, -1) circle (2pt) node[] (l){};
	\filldraw (0, -0.2) circle (2pt) node[] (m){};
	\draw [very thick, <-] (b) -- (a);
	\draw [very thick, <-] (b) -- (c);
	\draw [very thick, <-] (c) -- (d);
	\draw [very thick, <-] (e) -- (d);
	\draw [very thick, <-] (e) -- (f);
	\draw [very thick, <-] (g) -- (f);
	\draw [very thick, <-] (g) -- (h);
	\draw [very thick, <-] (i) -- (h);
	\draw [very thick, <-] (i) -- (j);
	\draw [very thick, <-] (k) -- (j);
	\draw [very thick, <-] (l) -- (k);
	\draw [very thick, <-] (l) -- (m);
	\draw [very thick, <-] (b) -- (e);
	\draw [very thick, <-] (a) -- (f);
	\draw [very thick, <-] (l) -- (i);
	\draw [very thick, <-] (m) -- (h);
	\node [above left, xshift=4pt] at (a) {$z_0$};
	\node [left] at (b) {$z_1$};
	\node [left] at (c) {$z_2$};
	\node [right] at (d) {$z_3$};
	\node [left] at (e) {$z_4$};
	\node [right] at (f) {$z_5$};
	\node [left] at (g) {$z_6$};
	\node [right] at (h) {$z_7$};
	\node [left, yshift=4pt] at (i) {$z_8$};
	\node [right] at (j) {$z_9$};
	\node [right] at (k) {$z_{10}$};
	\node [left] at (l) {$z_{11}$};
	\node [right, yshift=-3pt] at (m) {$z_{12}$};
	
\end{tikzpicture}
\caption{The relevant summand of $\CFKi(T_{2,9} \# -T_{2,3; 2, 5})$}
\label{fig:summand}
\end{figure}
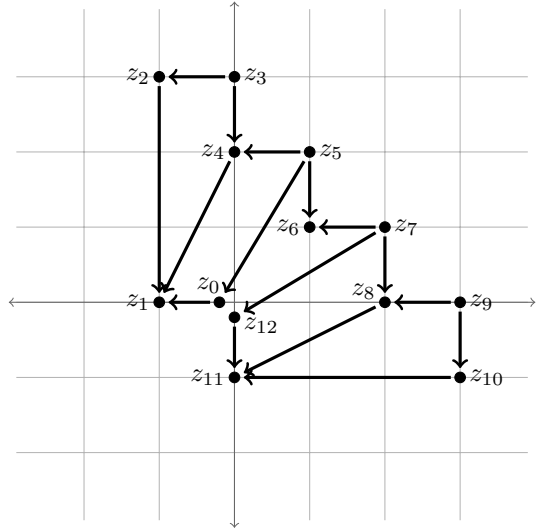

\begin{figure}[htb!]
\vspace{5pt}
\centering
\begin{tikzpicture}[scale=1]

	\draw[step=1, black!30!white, very thin] (-2.9, -2.9) grid (3.9, 3.9);
	
	\begin{scope}[thin, black!60!white]
		\draw [<->] (-3, 0) -- (4, 0);
		\draw [<->] (0, -3) -- (0, 4);
	\end{scope}
		
	\filldraw (-1.2, -1) circle (2pt) node[] (a){};
	\filldraw (-2, -1) circle (2pt) node[] (b){};
	\filldraw (-2, 2) circle (2pt) node[] (c){};
	\filldraw (-1, 2) circle (2pt) node[] (d){};
	\filldraw (-1, 1) circle (2pt) node[] (e){};
	\filldraw (0, 1) circle (2pt) node[] (f){};
	\filldraw (0, 0) circle (2pt) node[] (g){};
	\filldraw (1, 0) circle (2pt) node[] (h){};
	\filldraw (1, -1) circle (2pt) node[] (i){};
	\filldraw (2, -1) circle (2pt) node[] (j){};
	\filldraw (2, -2) circle (2pt) node[] (k){};
	\filldraw (-1, -2) circle (2pt) node[] (l){};
	\filldraw (-1, -1.2) circle (2pt) node[] (m){};
	\draw [very thick, <-] (b) -- (a);
	\draw [very thick, <-] (b) -- (c);
	\draw [very thick, <-] (c) -- (d);
	\draw [very thick, <-] (e) -- (d);
	\draw [very thick, <-] (e) -- (f);
	\draw [very thick, <-] (g) -- (f);
	\draw [very thick, <-] (g) -- (h);
	\draw [very thick, <-] (i) -- (h);
	\draw [very thick, <-] (i) -- (j);
	\draw [very thick, <-] (k) -- (j);
	\draw [very thick, <-] (l) -- (k);
	\draw [very thick, <-] (l) -- (m);
	\draw [very thick, <-] (b) -- (e);
	\draw [very thick, <-] (a) -- (f);
	\draw [very thick, <-] (l) -- (i);
	\draw [very thick, <-] (m) -- (h);
	%\node [above left, xshift=4pt] at (a) {$z_0$};
	%\node [left] at (b) {$z_1$};
	\node [left] at (c) {$U z_2$};
	\node [right] at (d) {$U z_3$};
	\node [left, yshift=2pt] at (e) {$U z_4$};
	\node [right] at (f) {$U z_5$};
	\node [below] at (g) {$U z_6$};
	\node [right] at (h) {$U z_7$};
	\node [left, yshift=4pt] at (i) {$U z_8$};
	\node [right] at (j) {$U z_9$};
	\node [right] at (k) {$U z_{10}$};
	%\node [left] at (l) {$z_{11}$};
	%\node [right, yshift=-3pt] at (m) {$z_{12}$};
	
	\filldraw[black!20!white, pattern=north west lines,] (-3, -3) rectangle (-0.15, 0.85);
	
\end{tikzpicture}
\caption{The generators $\{Uz_i\}$ in $A^+_1$}
\label{fig:A_1}
\end{figure}
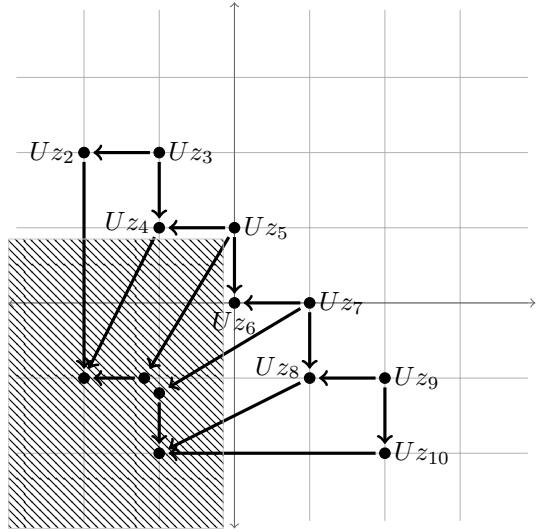

\begin{figure}[htb!]
\vspace{5pt}
\centering
\begin{tikzpicture}[scale=1]

	\draw[step=1, black!30!white, very thin] (-2.9, -2.9) grid (3.9, 3.9);
	
	\begin{scope}[thin, black!60!white]
		\draw [<->] (-3, 0) -- (4, 0);
		\draw [<->] (0, -3) -- (0, 4);
	\end{scope}
		
	\filldraw (-1.2, -1) circle (2pt) node[] (a){};
	\filldraw (-2, -1) circle (2pt) node[] (b){};
	\filldraw (-2, 2) circle (2pt) node[] (c){};
	\filldraw (-1, 2) circle (2pt) node[] (d){};
	\filldraw (-1, 1) circle (2pt) node[] (e){};
	\filldraw (0, 1) circle (2pt) node[] (f){};
	\filldraw (0, 0) circle (2pt) node[] (g){};
	\filldraw (1, 0) circle (2pt) node[] (h){};
	\filldraw (1, -1) circle (2pt) node[] (i){};
	\filldraw (2, -1) circle (2pt) node[] (j){};
	\filldraw (2, -2) circle (2pt) node[] (k){};
	\filldraw (-1, -2) circle (2pt) node[] (l){};
	\filldraw (-1, -1.2) circle (2pt) node[] (m){};
	\draw [very thick, <-] (b) -- (a);
	\draw [very thick, <-] (b) -- (c);
	\draw [very thick, <-] (c) -- (d);
	\draw [very thick, <-] (e) -- (d);
	\draw [very thick, <-] (e) -- (f);
	\draw [very thick, <-] (g) -- (f);
	\draw [very thick, <-] (g) -- (h);
	\draw [very thick, <-] (i) -- (h);
	\draw [very thick, <-] (i) -- (j);
	\draw [very thick, <-] (k) -- (j);
	\draw [very thick, <-] (l) -- (k);
	\draw [very thick, <-] (l) -- (m);
	\draw [very thick, <-] (b) -- (e);
	\draw [very thick, <-] (a) -- (f);
	\draw [very thick, <-] (l) -- (i);
	\draw [very thick, <-] (m) -- (h);
	%\node [above left, xshift=4pt] at (a) {$z_0$};
	%\node [left] at (b) {$z_1$};
	\node [left] at (c) {$U z_2$};
	\node [right] at (d) {$U z_3$};
	%\node [left, yshift=2pt] at (e) {$U z_4$};
	\node [right] at (f) {$U z_5$};
	\node [below] at (g) {$U z_6$};
	\node [right] at (h) {$U z_7$};
	\node [left, yshift=4pt] at (i) {$U z_8$};
	\node [right] at (j) {$U z_9$};
	\node [right] at (k) {$U z_{10}$};
	%\node [left] at (l) {$z_{11}$};
	%\node [right, yshift=-3pt] at (m) {$z_{12}$};
	
	\filldraw[black!20!white, pattern=north west lines,] (-3, -3) rectangle (-0.15, 1.85);
	
\end{tikzpicture}
\caption{The generators $\{Uz_i\}$ in $A^+_2$}
\label{fig:A_2}
\end{figure}
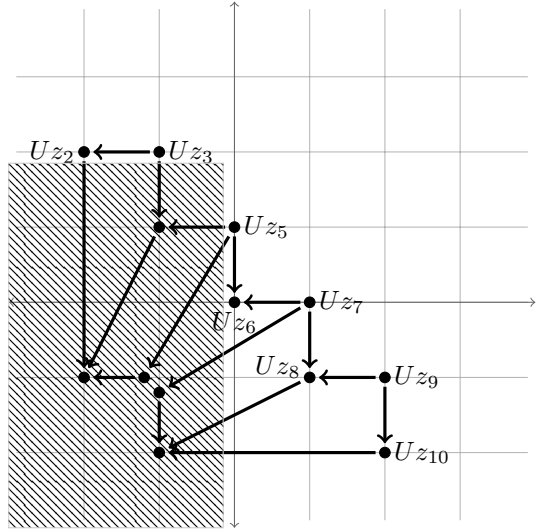

The cycle $z_6$ generates $H_*(C)$. Moreover, the cycle $Uz_6$ is non-zero in $H_*(A^+_1)$; see Figure \ref{fig:A_1}. The cycle $Uz_6$ is a boundary in $A^+_2$ as in Figure \ref{fig:A_2}, while the cycle $z_6$ is non-zero in $H_*(A^+_2)$. It follows that $\nu^+(T_{2,9} \# -T_{2,3; 2,5})=2$, as desired.
\end{proof}

\begin{cor}
Let $K=T_{2,5} \# 2 T_{2,3} \# -T_{2,3; 2,5}$. Then
\[ \tau(K) = 0, \qquad \nu(K) = 1, \quad \textup{ and } \quad \nu^+(K) = 2. \]
\end{cor}

\begin{proof}
By \cite[Theorem B.1]{HeddenKimLiv},
\[ \CFKi(T_{2,5} \# 2T_{2,3}) \cong \CFKi(T_{2,9}) \oplus A, \]
where $A$ is acyclic (i.e., its total homology vanishes). Since acyclic summands do not affect $\tau$, $\nu$, and $\nu^+$, the result follows.
\end{proof}

\begin{lem}\label{lem:4ballK}
Let $K=T_{2,5} \# 2 T_{2,3} \# -T_{2,3; 2,5}$. Then $g_4(K) = 2$.
\end{lem}

\begin{proof}
When $p, q >0$, the genus of $T_{p, q}$ is equal to $\frac{(p-1)(q-1)}{2}$. We can construct a genus $4$ Seifert surface $F$ for $-T_{2,3; 2,5}= (-T_{2,3})_{-2, 5}$ by taking two parallel copies of the genus one Seifert surface for $-T_{2,3}$ and connecting them with $5$ half-twisted bands. The knot $-T_{2,3} \# T_{-2, 5}$ sits on $F$. To see this, consider one copy of the Seifert surface for $-T_{2,3}$ together with the half-twisted bands and a small neighborhood of a segment connecting the ends of the bands.

Take the boundary sum of $F$ with the genus two Seifert surface for $T_{2,5}$ and with two copies of the genus one Seifert surface for $T_{2,3}$ to obtain a surface $F'$. The surface $F'$ is a genus $8$ Seifert surface for $K$. The genus $6$ slice knot $J=-T_{2,3} \# T_{-2, 5} \# T_{2,3} \# T_{2,5}$ sits on this surface. Performing surgery along $J$ on $F'$ in $B^4$ yields a genus two slice surface for $K$. Since $\nu^+(K) = 2$ and $\nu^+(K) \leq g_4(K)$, it follows that $g_4(K)=2$.
\end{proof}

In order to prove the main theorem, we will consider certain cables of the knot $K=T_{2,5} \# 2 T_{2,3} \# -T_{2,3; 2,5}$. We first compute $\tau$ of these cables.

\begin{lem}
Let $K$ be the knot $T_{2,5} \# 2 T_{2,3} \# -T_{2,3; 2,5}$.  Then $$\tau(K_{p,3p-1})=\frac{3p(p-1)}{2}.$$
\end{lem}

\begin{proof}
Recall from \cite[Definition 3.4]{Homcables} that the invariant $\varepsilon(K)$ is defined to be $-1$ if $\tau(K)<\nu(K)$. The equality then follows from \cite[Theorem 1]{Homcables}, which states that if $\varepsilon(K)=-1$, then
\[ \tau(K_{p,q}) = p\tau(K) + \frac{(p-1)(q+1)}{2}. \]
\end{proof}

\begin{prop} \label{prop:nuexamples}

Let $K$ be the knot $T_{2,5} \# 2 T_{2,3} \# -T_{2,3; 2,5}$.  Then $$\nu^+(K_{p,3p-1})=g_4(K_{p, 3p-1})=\frac{p(3p-1)}{2}+1.$$

\end{prop}

\begin{proof}

Let $p, q >0$. For a cable knot $K_{p,q}$, there is a reducible surgery $$S^3_{pq}(K_{p,q})\cong S^3_{q/p}(K) \# L(p,q). $$
We apply the surgery formula (\ref{Corr}) for the above knot surgery when $K$ is the unknot. Note that $\max \{V_i, H_{i-pq}\} = V_i$ when $0 \leq i \leq \frac{pq}{2}$ since $V_i = H_{-i}$ and $H_{i-1} \leq H_i$. Thus, we have
\begin{equation}\label{eq1}
d(L(pq, 1),i)-2V_i(T_{p,q})=d(L(q,p), p_1(i))+d(L(p,q), p_2(i))
\end{equation}
for all $0 \leq i\leq \frac{pq}{2}$.

Here, we identify the \spinc structure of a rational homology sphere by an integer $i$ as in \cite{NiWu}, and $p_1(i)$ and $p_2(i)$ are the projection of the \spinc structure to the two factors of the reducible manifold.  In particular, we can identify $p_1(i)$ with some integers between 0 and $q-1$ and $p_2(i)$ with some integers between 0 and $p-1$.

We can also apply (\ref{Corr}) for an arbitrary knot $K$.  We have

\begin{eqnarray*}\label{eq2}
d(L(pq, 1),i)-2V_i(K_{p,q})&=&d(L(q,p), p_1(i))
-2\max\{V_{\lfloor\frac{p_1(i)}p\rfloor}(K),H_{\lfloor\frac{p_1(i)-q}p\rfloor}(K)\}\\ && +d(L(p,q), p_2(i)).
\end{eqnarray*}
for all $i\leq \frac{pq}{2}$.

Compared with Equation (\ref{eq1}) and using the fact $V_i(T_{p,q})\geq0$, we deduce that for all $i\leq \frac{pq}{2}$,
\begin{eqnarray*}
V_i(K_{p,q})&=&V_i(T_{p,q})+\max\{V_{\lfloor\frac{p_1(i)}p\rfloor}(K),H_{\lfloor\frac{p_1(i)-q}p\rfloor}(K)\}\\
    &\geq&\max\{V_{\lfloor\frac{p_1(i)}p\rfloor}(K),H_{\lfloor\frac{p_1(i)-q}p\rfloor}(K)\}
\end{eqnarray*}

From now on, let us specialize to the case when $K$ is the knot $T_{2,5} \# 2 T_{2,3} \# -T_{2,3; 2,5}$ and $q=3p-1$. We claim that
$$\max\{V_{\lfloor\frac{p_1(i)}p\rfloor},H_{\lfloor\frac{p_1(i)-q}p\rfloor}\}>0.$$  To see this, note that $V_0(K),V_1(K) >0$ as $\nu^+(K)=2$.  When $0\leq p_1(i) <2p$, $V_{\lfloor\frac{p_1(i)}p\rfloor}(K)>0$.
Otherwise, $2p\leq p_1(i) <q=3p-1$, and then $H_{\lfloor\frac{p_1(i)-q}p\rfloor}(K)>0$ since $H_{-k} = V_k$ and $V_0(K), V_1(K) > 0$.
%This proves the positivity of $\max\{V_{\lfloor\frac{p_1(i)}p\rfloor},H_{\lfloor\frac{p_1(i)-q}p\rfloor}\}$.

Hence, $V_i(K_{p,q})>0$ for all $i\leq \frac{pq}{2}$.  This implies that
$$\nu^+(K_{p,3p-1}) \geq \frac{p(3p-1)}{2}+1.$$
On the other hand,
\[ g_4(K_{p, q}) \leq pg_4(K)+ \frac{(p-1)(q-1)}{2}, \]
since one can construct a slice surface for $K_{p,q}$ from $p$ parallel copies of a slice surface for $K$ together with $(p-1)q$ half-twisted bands. By Lemma \ref{lem:4ballK}, $g_4(K)=2$, so when $q=3p-1$, the right-hand side of the above inequality is $\frac{p(3p-1)}{2}+1$. Hence
\[ \frac{p(3p-1)}{2}+1 \leq \nu^+(K_{p, 3p-1}) \leq g_4(K_{p, 3p-1}) \leq \frac{p(3p-1)}{2}+1, \]
so $\nu^+(K_{p, 3p-1}) = g_4(K_{p, 3p-1}) = \frac{p(3p-1)}{2}+1$.
\end{proof}

Note that $\nu^+(K_{p,3p-1})-\tau(K_{p,3p-1})=p+1$ for $K=T_{2,5} \# 2 T_{2,3} \# -T_{2,3; 2,5}$.  This proves Theorem \ref{betterbound}.

A similar argument shows that $\nu^+$ gives a sharp four-ball genus bound for certain other cable knots as well.

\begin{prop}
Let $K$ be a knot with $\nu^+(K)=g_4(K)=n$, then
$$\nu^+(K_{p, (2n-1)p-1})=g_4(K_{p, (2n-1)p-1})=\frac{p((2n-1)p-1)}{2}+1.$$

\end{prop}

\begin{proof}
 Let $q=(2n-1)p-1$.  We proved $$V_i(K_{p,q}) \geq \max\{V_{\lfloor\frac{p_1(i)}p\rfloor}(K),H_{\lfloor\frac{p_1(i)-q}p\rfloor}(K)\}.$$
We claim that
$$\max\{V_{\lfloor\frac{p_1(i)}p\rfloor},H_{\lfloor\frac{p_1(i)-q}p\rfloor}\}>0.$$  To see this, note that $V_i(K) >0$ for all $i<n$.  When $0\leq p_1(i) <np$, $V_{\lfloor\frac{p_1(i)}p\rfloor}(K)>0$.
Otherwise, $np\leq p_1(i) <q=(2n-1)p-1$, and then $H_{\lfloor\frac{p_1(i)-q}p\rfloor}(K)>0$.  Hence, $V_i(K_{p,q})>0$ for all $i\leq \frac{pq}{2}$.  This implies that
\begin{eqnarray*}
\nu^+(K_{p,q}) &\geq& \frac {pq}{2}+1 \\
&=&  \frac{p((2n-1)p-1)}{2}+1.
\end{eqnarray*}
On the other hand,
\begin{eqnarray*}
 g_4(K_{p, q}) &\leq& pg_4(K)+ \frac{(p-1)(q-1)}{2}\\
&=&pn+\frac{(p-1)((2n-1)p-2)}{2}\\
&=&\frac{p((2n-1)p-1)}{2}+1.
\end{eqnarray*}
So $\nu^+(K_{p, (2n-1)p-1})=g_4(K_{p, (2n-1)p-1})=\frac{p((2n-1)p-1)}{2}+1.$
\end{proof}

We conclude by showing that the knot signature cannot detect the four-ball genus of the knots used in Theorem \ref{betterbound}. Recall that
\[ \frac{1}{2} |\sigma(K)| \leq g_4(K). \]

\begin{prop} \label{prop:sig}
Let $K = T_{2,5} \# 2T_{2,3} \# -T_{2,3; 2,5}$. Then for $p>0$,
\[ \frac{1}{2}|\sigma(K_{p, 3p-1})| +2p-2 \leq g_4(K_{p, 3p-1}). \]
\end{prop}

\begin{proof}
We have that $\sigma(T_{2, q}) = 1-q$. By \cite[Theorem 9]{Shinohara},
\[
	\sigma(K_{p,q})=\left\{
	\begin{array}{ll}
	\sigma(T_{p,q}) &\textup{if } p \textup{ is even}\\
	\sigma(K)+\sigma(T_{p,q}) &\textup{if }p \textup{ is odd}.
	\end{array}
	\right.
\]
Thus, $\sigma(T_{2,3; 2,5}) = -4$ and since signature is additive under connected sum,
\begin{align*}
	\sigma(T_{2,5} \# 2T_{2,3} \# -T_{2,3; 2,5}) &= -4 + 2(-2) - (-4) \\
		&= -4.
\end{align*}
We showed in Lemma \ref{lem:4ballK} that $g_4(K) = 2$, so for $K$, the signature is indeed strong enough to detect the four-ball genus. However, we will now show that it is not strong enough to detect the four-ball genus of $K_{p, 3p-1}$. We have that
\begin{align*}
	|\sigma(K_{p, 3p-1})| &\leq |\sigma(K)| + |\sigma(T_{p,3p-1})| \\
		 &\leq 4 + (p-1)(3p-2) \\
		 &= 3p^2-5p+6,
\end{align*}
where the second inequality follows from the fact that when $p, q>0$,
\[ |\sigma(T_{p,q})| \leq 2g_4(T_{p,q}) = (p-1)(q-1) . \]
On the other hand,
\[ 2g_4(K_{p, 3p-1}) = 3p^2 -p +2, \]
so
\[ |\sigma(K_{p, 3p-1})| +4p-4 \leq 2g_4(K_{p, 3p-1}). \]
\end{proof}

Recall from Proposition \ref{prop:nuexamples} that $g_4(K_{p, 3p-1}) = \nu^+(K_{p, 3p-1})$. A consequence of Proposition \ref{prop:sig} is that the gap between $\frac{1}{2}\sigma$ and $\nu^+$ can be made arbitrarily large.

\bibliographystyle{amsalpha}
\bibliography{mybib}

\end{document}